\documentclass[leqno,12pt]{article} %leqno is the option to put formula numbers on the left side
\usepackage{amsmath, amssymb}
\usepackage{amsthm} %theorem environment option
\theoremstyle{plain} %text of this environment is typesetted in italics
\newtheorem{theorem}{\indent\sc Theorem}[section]
\newtheorem{lemma}[theorem]{\indent\sc Lemma}
\newtheorem{corollary}[theorem]{\indent\sc Corollary}
\newtheorem{proposition}[theorem]{\indent\sc Proposition}

\theoremstyle{definition} %text of this environment is typesetted in roman letters

\newtheorem{remark}[theorem]{\indent\sc Remark}
\newtheorem{example}[theorem]{\indent\sc Example}

\newcommand\on{\operatorname}

\newcommand\sign{\on{sign}}

\title{Certain symmetries of $\mathbb R^2$ with diagonal metrics}
\author{Adara M. Blaga}

\date{}
\begin{document}

\maketitle

\markboth{{\small\it {\hspace{1cm} Certain symmetries of $\mathbb R^2$ with diagonal metrics}}}{\small\it{Certain symmetries of $\mathbb R^2$ with diagonal metrics\hspace{1cm}}}

%%%%%%%%%%%%%%% footnote %%%%%%%%%%%%%%%%
\footnote{ %2010 MSC numbers
2010 \textit{Mathematics Subject Classification}.
35Q51, 53B25, 53B50.
}
\footnote{ %key words and phrases
\textit{Key words and phrases}.
diagonal metric; Killing vector field.
}

\begin{abstract}
We put into light the Killing vector fields on $\mathbb R^2$ endowed with a diagonal Riemannian metric. According to certain restrictions on the Lam\'{e} coefficients, we describe the symmetries of the metric.
\end{abstract}

\section{Preliminaries}

Killing vector fields represent a special class of vector fields on a Riemannian manifold, constituting the symmetries of the metric. Their role in understanding the geometry of the space is crucial, helping to identify the isometries. In general relativity they are used to describe the symmetries of the spacetime such as time translation and rotations. They lead to conserved quantities along the geodesics, such as the energy and the angular momentum, which are important in physical theories.
More about the topological consequences of the existing of Killing vector fields on a Riemannian manifold, as well as their applications, can be found, for instance, in \cite{ber,be,vb,g,no,ro,yo}.

The aim of this paper is to characterize the Killing vector fields on $\mathbb R^2$ with respect to a diagonal metric
$\left( \begin{array}{cc}
    \frac{1}{f_1^2} & 0 \\
    0 & \frac{1}{f_2^2} \\
  \end{array}
\right)
$, where $f_1$ and $f_2$ are smooth real-valued functions nowhere zero on $\mathbb R^2$, when the Lam\'{e} coefficients depend on a single variable. Moreover, if one of them is a constant, we will determine the Killing vector fields with respect to that metric. By fixing a metric, there exist three independent Killing vector fields on $\mathbb R^2$, namely, corresponding to two translations and one rotation. 
Recently, some properties of $\mathbb R^3$ endowed with a similar metric, such as, the flatness of the manifold and the existence of almost $\eta$-Ricci solitons have been studied in \cite{bl22} and \cite{balr}.

\bigskip

Let ${g}$ be a Riemannian metric on $\mathbb R^2$ defined by:
$${g}=\frac{1}{f_1^2}dx^1\otimes dx^1+\frac{1}{f_2^2}dx^2\otimes dx^2,$$
where $f_1$ and $f_2$ are smooth real functions nowhere zero on $\mathbb R^2$, and $x^1,x^2$ are the standard coordinates in $\mathbb R^2$. The functions $\beta_i=\frac{\displaystyle 1}{\displaystyle f_i^2}$, $i\in \{1,2\}$, are known as the \textit{Lam\'{e} coefficients} of the metric. Let
$$\Big\{E_1:=f_1\frac{\partial}{\partial x^1}, \ \ E_2:=f_2\frac{\partial}{\partial x^2}\Big\}$$
be a local orthonormal frame w.r.t. $g$. We will denote as follows:
$$\frac{f_2}{f_1}\cdot\frac{\partial f_1}{\partial x^2}=:l_{12}, \ \ \frac{f_1}{f_2}\cdot\frac{\partial f_2}{\partial x^1}=:l_{21}.$$
On the basis vector fields $E_1$ and $E_2$, the Lie bracket is given by:
$$[E_1,E_2]=-l_{12}E_1+l_{21}E_2=-[E_2,E_1],$$
and the Levi-Civita connection $\nabla$ of $g$ is given by:
%(see \cite{balr}):
$$\nabla_{E_1}E_1=l_{12}E_2, \ \ \nabla_{E_1}E_2=-l_{12}E_1, \ \ \nabla_{E_2}E_2=l_{21}E_1,  \ \ \nabla_{E_2}E_1=-l_{21}E_2.$$

Whenever a function $f$ on $\mathbb R^2$ depends on one of its variables, we will write in its argument only that variable in order to emphasize this fact, for example, $f(x^i)$.

\section{Killing vector fields w.r.t. these metrics}

We recall that a vector field $V$ is called a \textit{Killing vector field} with respect to a Riemannian metric $g$ (or, a \textit{symmetry} of $g$) \cite{killing} if the Lie derivative $\pounds$ of $g$ in the direction of $V$ vanishes, i.e., if the Killing's equation, $\pounds_Vg=0$, holds true. In this case, $g$ is preserved by the flow of $V$ (which consists of isometries).

\bigskip

Let $V=V^1E_1+V^2E_2$, where $V^1$ and $V^2$ are two smooth functions on $\mathbb R^2$. Then,
\begin{align*}
(\pounds_Vg)(E_i,E_j):&=V(g(E_i,E_j))-g([V,E_i],E_j)-g(E_i,[V,E_j])\\
&=g(\nabla_{E_i}V,E_j)+g(E_i,\nabla_{E_j}V)\\
&=E_i(V^j)+E_j(V^i)+\sum_{k=1}^2V^k\{g(\nabla_{E_i}E_k,E_j)+g(E_i,\nabla_{E_j}E_k)\}
\end{align*}
for any $i,j\in \{1,2\}$, which is equivalent to
\begin{equation*}
\left\{
    \begin{aligned}
&(\pounds_Vg)(E_1,E_1)=2\{E_1(V^1)-l_{12}V^2\}\\
&(\pounds_Vg)(E_2,E_2)=2\{E_2(V^2)-l_{21}V^1\}\\
&(\pounds_Vg)(E_1,E_2)=E_1(V^2)+E_2(V^1)+l_{12}V^1+l_{21}V^2
    \end{aligned}
  \right. \ ,
\end{equation*}
and we can state

\begin{lemma}\label{lp}
A vector field $V=V^1E_1+V^2E_2$ is a Killing vector field on $(\mathbb R^2,g)$ if and only if
\begin{equation*}
\left\{
    \begin{aligned}
&f_1\frac{\displaystyle \partial V^1}{\displaystyle \partial x^1}-\frac{\displaystyle f_2}{\displaystyle f_1}\frac{\displaystyle \partial f_1}{\displaystyle \partial x^2}V^2=0\\
&f_2\frac{\displaystyle \partial V^2}{\displaystyle \partial x^2}-\frac{\displaystyle f_1}{\displaystyle f_2}\frac{\displaystyle \partial f_2}{\displaystyle \partial x^1}V^1=0\\
&f_1\frac{\displaystyle \partial V^2}{\displaystyle \partial x^1}+f_2\frac{\displaystyle \partial V^1}{\displaystyle \partial x^2}+\frac{\displaystyle f_2}{\displaystyle f_1}\frac{\displaystyle \partial f_1}{\displaystyle \partial x^2}V^1+\frac{\displaystyle f_1}{\displaystyle f_2}\frac{\displaystyle \partial f_2}{\displaystyle \partial x^1}V^2=0
    \end{aligned}
  \right. \ .
\end{equation*}
\end{lemma}

A natural question which arise is: \textit{When the orthonormal basis vector fields are Killing vector fields for the metric?}

\begin{proposition}
Let $i\in \{1,2\}$. The vector field $E_i=f_i\frac{\displaystyle \partial}{\displaystyle \partial x^i}$ is a Killing vector field on $(\mathbb R^2,g)$ if and only if $f_1$ depends only on $x^1$ and $f_2$ depends only on $x^2$.
\end{proposition}
\begin{proof}
In this case, the condition of being a Killing vector field from Lemma \ref{lp} becomes
\begin{equation*}
\left\{
    \begin{aligned}
&\frac{\displaystyle \partial f_2}{\displaystyle \partial x^1}=0\\
&\frac{\displaystyle \partial f_1}{\displaystyle \partial x^2}=0
    \end{aligned}
  \right. \ ,
\end{equation*}
hence, $f_1=f_1(x^1)$ and $f_2=f_2(x^2)$. 
\end{proof}

Concerning the orthogonal vector fields $\frac{\displaystyle \partial }{\displaystyle \partial x^i}$, $i\in \{1,2\}$, we prove the following.

\begin{proposition}
Let $i\in \{1,2\}$. The vector field $\frac{\displaystyle \partial}{\displaystyle \partial x^i}$ is a Killing vector field on $(\mathbb R^2,g)$ if and only if $f_1$ and $f_2$ depend only on $x^j$, where $j\in \{1,2\}$, $j\neq i$.
\end{proposition}
\begin{proof}
In this case, the condition of being a Killing vector field from Lemma \ref{lp} becomes
\begin{equation*}
\left\{
    \begin{aligned}
&\frac{\displaystyle \partial f_1}{\displaystyle \partial x^i}=0\\
&\frac{\displaystyle \partial f_2}{\displaystyle \partial x^i}=0
    \end{aligned}
  \right. \ ,
\end{equation*}
hence, $f_1=f_1(x^j)$ and $f_2=f_2(x^j)$, where $j\in \{1,2\}$, $j\neq i$.
\end{proof}

Now, if we assume that both of the Lam\'{e} coefficients depend on a single variable, we get the following characterizations of a Killing vector field.

\begin{proposition}
Let $V=V^1E_1+V^2E_2$, with $V^1$ and $V^2$ smooth functions on $\mathbb R^2$.

(i) If $f_1=f_1(x^1)$ and $f_2=f_2(x^2)$, then $V$ is a Killing vector field on $(\mathbb R^2,g)$ if and only if
\begin{equation*}
\left\{
    \begin{aligned}
&\frac{\displaystyle \partial V^1}{\displaystyle \partial x^1}=0\\
&\frac{\displaystyle \partial V^2}{\displaystyle \partial x^2}=0\\
&f_1\frac{\displaystyle \partial V^2}{\displaystyle \partial x^1}+f_2\frac{\displaystyle \partial V^1}{\displaystyle \partial x^2}=0
\end{aligned}
  \right. \ .
\end{equation*}

(ii) If $f_1=f_1(x^1)$ and $f_2=f_2(x^1)$, then $V$ is a Killing vector field on $(\mathbb R^2,g)$ if and only if
\begin{equation*}
\left\{
    \begin{aligned}
&\frac{\displaystyle \partial V^1}{\displaystyle \partial x^1}=0\\
&f_2\frac{\displaystyle \partial V^2}{\displaystyle \partial x^2}-f_1\frac{\displaystyle f_2'}{\displaystyle f_2}V^1=0\\
&f_1\frac{\displaystyle \partial V^2}{\displaystyle \partial x^1}+f_2\frac{\displaystyle \partial V^1}{\displaystyle \partial x^2}+f_1\frac{\displaystyle f_2'}{\displaystyle f_2}V^2=0
\end{aligned}
  \right. \ .
\end{equation*}

(iii) If $f_1=f_1(x^2)$ and $f_2=f_2(x^1)$, then $V$ is a Killing vector field on $(\mathbb R^2,g)$ if and only if
\begin{equation*}
\left\{
    \begin{aligned}
&f_1\frac{\displaystyle \partial V^1}{\displaystyle \partial x^1}-f_2\frac{\displaystyle f_1'}{\displaystyle f_1}V^2=0\\
&f_2\frac{\displaystyle \partial V^2}{\displaystyle \partial x^2}-f_1\frac{\displaystyle f_2'}{\displaystyle f_2}V^1=0\\
&f_1\frac{\displaystyle \partial V^2}{\displaystyle \partial x^1}+f_2\frac{\displaystyle \partial V^1}{\displaystyle \partial x^2}+f_2\frac{\displaystyle f_1'}{\displaystyle f_1}V^1+f_1\frac{\displaystyle f_2'}{\displaystyle f_2}V^2=0
\end{aligned}
  \right. \ .
\end{equation*}
\end{proposition}
\begin{proof}
The conclusions follow immediately from Lemma \ref{lp}.
\end{proof}

We shall further consider the case when the component function $V^i$ of a vector field $V$ depends on the same variable as the function $f_i$ does, for $i\in \{1,2\}$, and using the above proposition, we prove the following results.

\begin{theorem}
If $f_i=f_i(x^i)$ and $V^i=V^i(x^i)$, for $i\in \{1,2\}$, then $V=V^1E_1+V^2E_2$ is a Killing vector field on $(\mathbb R^2,g)$ if and only if $V^i=c_i\in \mathbb R$, for $i\in \{1,2\}$.
\end{theorem}
\begin{proof}
We have
$$(V^1)'=(V^2)'=0,$$
and we get the conclusion.
\end{proof}

\begin{example}
$V=e^{x^1}\frac{\displaystyle \partial}{\displaystyle \partial x^1}+e^{x^2}\frac{\displaystyle \partial}{\displaystyle \partial x^2}$ is a Killing vector field on
$$\left(\mathbb R^2, \ g=e^{-2x^1}dx^1\otimes dx^1+e^{-2x^2}dx^2\otimes dx^2\right).$$
\end{example}

\begin{theorem}
If $f_i=f_i(x^1)$ and $V^i=V^i(x^1)$, for $i\in \{1,2\}$, then $V=V^1E_1+V^2E_2$ is a Killing vector field on $(\mathbb R^2,g)$ if and only if one of the following assertions holds:

(i) $V^1=c_1\in \mathbb R\setminus \{0\}$, $V^2=c_2\in \mathbb R$, and $f_2$ is constant;

(ii) $V^1=0$ and $V^2=\frac{\displaystyle c_2}{\displaystyle f_2}$, where $c_2\in \mathbb R$.
\end{theorem}
\begin{proof}
We have
\begin{equation*}
\left\{
    \begin{aligned}
&(V^1)'=0\\
&f_2'V^1=0\\
&f_2(V^2)'+f_2'V^2=0
\end{aligned}
  \right. \ ,
\end{equation*}
and we get the conclusion. 

The converse implication also holds true.
\end{proof}

\begin{example}
$V=e^{x^1}\frac{\displaystyle \partial}{\displaystyle \partial x^1}+\frac{\displaystyle \partial}{\displaystyle \partial x^2}$ is a Killing vector field on
$$\left(\mathbb R^2, \ g=e^{-2x^1}dx^1\otimes dx^1+dx^2\otimes dx^2\right).$$
\end{example}

\begin{example}
$V=\frac{\displaystyle \partial}{\displaystyle \partial x^2}$ is a Killing vector field on
$$\left(\mathbb R^2, \ g=e^{2x^1}dx^1\otimes dx^1+e^{-2x^1}dx^2\otimes dx^2\right).$$
\end{example}

\begin{theorem}\label{pd}
If $f_1=f_1(x^2)$, $f_2=f_2(x^1)$, $V^1=V^1(x^2)$, $V^2=V^2(x^1)$, then $V$ is a Killing vector field on $(\mathbb R^2,g)$ if and only if
\begin{equation*}
\left\{
    \begin{aligned}
&V^1(x^2)=\frac{kF_1(x^2)+c_1}{f_1(x^2)}\\
&V^2(x^1)=-\frac{kF_2(x^1)+c_2}{f_2(x^1)}\\
&f_1'V^2=f_2'V^1=0
\end{aligned}
  \right. \ ,
\end{equation*}
where $k,c_1,c_2\in \mathbb R$ and $F_i'=f_i^2$, for $i\in \{1,2\}$.
\end{theorem}
\begin{proof}
We have
\begin{equation*}
\left\{
    \begin{aligned}
&f_1'V^2=0\\
&f_2'V^1=0\\
&f_1(V^2)'+f_2(V^1)'+f_2\frac{\displaystyle f_1'}{\displaystyle f_1}V^1+f_1\frac{\displaystyle f_2'}{\displaystyle f_2}V^2=0
\end{aligned}
  \right. \ .
\end{equation*}
The last equation of the system is equivalent to
$$f_1^2(f_2V^2)'+f_2^2(f_1V^1)'=0$$
and further to
$$\frac{\displaystyle (f_1V^1)'}{\displaystyle f_1^2}=-\frac{\displaystyle (f_2V^2)'}{\displaystyle f_2^2}.$$
Since the left term from the previous equality depends only on $x^2$ and the right term depends only on $x^1$, we deduce that the ratio must be a constant, let's say $k$, and we get, by integration,
$$V^1(x^2)=\frac{kF_1(x^2)+c_1}{f_1(x^2)} \ \ \textrm{and} \ \ V^2(x^1)=-\frac{kF_2(x^1)+c_2}{f_2(x^1)},$$
where $c_1,c_2\in \mathbb R$ and $F_i'=f_i^2$, for $i\in \{1,2\}$. 

The converse implication also holds true.
\end{proof}

\begin{remark}
Under the hypotheses of Theorem \ref{pd}, the condition $f_1'V^2=f_2'V^1=0$ is equivalent to the fact that one of the following cases holds true:

(1) $f_1=k_1\in \mathbb R\setminus \{0\}$ and $f_2=k_2\in \mathbb R\setminus \{0\}$, or

(2) $f_1=k_1\in \mathbb R\setminus \{0\}$ and $V^1=0$, or

(3) $f_2=k_2\in \mathbb R\setminus \{0\}$ and $V^2=0$, or

(4) $V^1=V^2=0$. 

In the first case, we get:

\begin{equation*}
\left\{
    \begin{aligned}
&V^1(x^2)=kk_1x^2+c_1\\
&V^2(x^1)=-kk_2x^1+c_2
\end{aligned}
  \right. \ , \ \ k, c_1,c_2\in \mathbb R.
\end{equation*}
\end{remark}

If the vector field has a zero component function or if one of the Lam\'{e} coefficients of the metric is constant, we obtain the following consequences.
\begin{corollary}
Under the hypotheses of Theorem \ref{pd}, if $V$ is a Killing vector field and $V^i=0$, then, either $V$ is the zero vector field or $f_i$ is constant.
\end{corollary}
\begin{proof}
In this case,
\begin{equation*}
\left\{
    \begin{aligned}
&f_i'V^j=0\\
&(f_jV^j)'=0
\end{aligned}
  \right. \ ,
\end{equation*}
for $j\neq i$. From the second equation of this system we find
$V^j=\frac{\displaystyle c_j}{\displaystyle f_j}$, where $c_j\in \mathbb R$, which, combined with the first one, gives $f_i'c_j=0$. Therefore, $c_j=0$ (hence, $V^j=0$) or $f_i$ is constant.
\end{proof}

\begin{example}
$V=\frac{\displaystyle \partial}{\displaystyle \partial x^1}$ is a Killing vector field on
$$\left(\mathbb R^2, \ g=e^{-2x^2}dx^1\otimes dx^1+dx^2\otimes dx^2\right).$$
\end{example}

\begin{corollary}
Under the hypotheses of Theorem \ref{pd}, if $V$ is a Killing vector field and $f_i=k_i\in \mathbb R\setminus \{0\}$, we get
$$V^i(x^j)=c_1x^j+c_2 , \ \ V^j(x^i)=-\frac{\frac{\displaystyle c_1}{\displaystyle k_i}F_j(x^i)+c_3}{f_j(x^i)},\ \ \textrm{and} \ \ f_j'V^i=0,$$
for $j\neq i$, where $c_1,c_2,c_3\in \mathbb R$ and $F_j'=f_j^2$, which is equivalent to
\begin{equation*}
\left\{
    \begin{aligned}
&V^i=0\\
&V^j(x^i)=\frac{\displaystyle c_j}{\displaystyle f_j(x^i)}
\end{aligned}
  \right. \ , \ \  c_j\in \mathbb R, \ \ or
\end{equation*}
\begin{equation*}
\left\{
    \begin{aligned}
&f_j=k_j\\
&V^i(x^j)=c_1x^j+c_2\\
&V^j(x^i)=-\frac{\displaystyle c_1k_j}{\displaystyle k_i}x^i+c_3
\end{aligned}
  \right. \ , \ \  k_j\in \mathbb R\setminus\{0\}, c_1,c_2,c_3\in \mathbb R.
\end{equation*}
\end{corollary}
\begin{proof}
It follows immediately from the expression of $V^i$.
\end{proof}

\begin{example}
$V=k_1^2x^2\frac{\displaystyle \partial}{\displaystyle \partial x^1}-k_2^2x^1\frac{\displaystyle \partial}{\displaystyle \partial x^2}$ is a Killing vector field on
$$\left(\mathbb R^2, \ g=\frac{\displaystyle 1}{\displaystyle k_1^2} dx^1\otimes dx^1+\frac{\displaystyle 1}{\displaystyle k_2^2} dx^2\otimes dx^2\right).$$
\end{example}

If one of the functions, let's say $f_2$, is constant, and if we do not impose any restriction on $V^i$, $i\in \{1,2\}$, then we obtain the expressions of the Killing vector fields as follows.

\begin{theorem}
If $f_1=f_1(x^1)$ and $f_2=k_2\in \mathbb R\setminus \{0\}$, then $V=V^1E_1+V^2E_2$ is a Killing vector field on $(\mathbb R^2,g)$ if and only if
$$V^1(x^2)=-\frac{k}{k_2}x^2+c_1 \ \ \textrm{and} \ \ V^2(x^1)=kF_2(x^1)+c_2,$$
where $k\in \mathbb R$, $c_1,c_2\in \mathbb R$, and $F_2'=\frac{\displaystyle 1}{\displaystyle f_1}$.
\end{theorem}
\begin{proof}
We have
\begin{equation*}
\left\{
    \begin{aligned}
&\frac{\displaystyle \partial V^1}{\displaystyle \partial x^1}=0\\
&\frac{\displaystyle \partial V^2}{\displaystyle \partial x^2}=0\\
&f_1\frac{\displaystyle \partial V^2}{\displaystyle \partial x^1}=-k_2\frac{\displaystyle \partial V^1}{\displaystyle \partial x^2}
\end{aligned}
  \right. \ .
\end{equation*}
It follows that
$$V^1=V^1(x^2), \ \ V^2=V^2(x^1), \ \ \textrm{and} \ \ f_1(V^2)'=-k_2(V^1)'= \textrm{constant},$$
and we get the conclusion. 

The converse implication also holds true.
\end{proof}

\begin{example}
$V=x^2e^{x^1}\frac{\displaystyle \partial}{\displaystyle \partial x^1}+e^{-x^1}\frac{\displaystyle \partial}{\displaystyle \partial x^2}$ is a Killing vector field on
$$\left(\mathbb R^2, \ g=e^{-2x^1}dx^1\otimes dx^1+dx^2\otimes dx^2\right).$$
\end{example}

\begin{theorem}\label{g6}
If $f_1=f_1(x^2)$ and $f_2=k_2\in \mathbb R\setminus \{0\}$, then $V=V^1E_1+V^2E_2$ is a Killing vector field on $(\mathbb R^2,g)$ if and only if
$$V^1(x^1,x^2)=-\frac{\displaystyle F_1(x^2)}{\displaystyle k_2f_1(x^2)}(V^2)'(x^1)+\frac{\displaystyle F(x^1)}{\displaystyle f_1(x^2)},$$
where $F_1'=f_1^2$, $F'(x^1)=k_2\displaystyle \frac{f_1'(x^2)}{f_1(x^2)}V^2(x^1)+\displaystyle \frac{F_1(x^2)}{k_2}(V^2)''(x^1)$, and $V^2$ has one of the following expressions:

(i) $V^2(x^1)=c_1x^1+c_2$ and $f_1(x^2)=c_3e^{c_4x^2}$, where $c_1,c_2,c_4\in \mathbb R$ and $c_3\in \nolinebreak \mathbb R\setminus \nolinebreak \{0\}$;

(ii) $V^2(x^1)=c_1e^{\sqrt{k}x^1}+c_2e^{-\sqrt{k}x^1}$, where $c_1,c_2\in \mathbb R$ and $k:=-k_2^2\frac{\displaystyle1}{\displaystyle f_1^2}\left(\frac{\displaystyle f_1'}{\displaystyle f_1}\right)'>0$ is a constant;

(iii) $V^2(x^1)=c_1\cos({\sqrt{-k}x^1})+c_2\sin({\sqrt{-k}x^1})$, where $c_1,c_2\in \mathbb R$ and $k:=-k_2^2\frac{\displaystyle1}{\displaystyle f_1^2}\left(\frac{\displaystyle f_1'}{\displaystyle f_1}\right)'< \nolinebreak 0$ is a constant.
\end{theorem}
\begin{proof}
We have
\begin{equation*}
\left\{
    \begin{aligned}
&\frac{\displaystyle \partial V^1}{\displaystyle \partial x^1}=k_2\frac{\displaystyle f_1'}{\displaystyle f_1^2}V^2\\
&\frac{\displaystyle \partial V^2}{\displaystyle \partial x^2}=0\\
&\frac{\displaystyle \partial V^2}{\displaystyle \partial x^1}=-k_2\frac{\displaystyle 1}{\displaystyle f_1}\frac{\displaystyle \partial V^1}{\displaystyle \partial x^2}-k_2\frac{\displaystyle f_1'}{\displaystyle f_1^2}V^1
\end{aligned}
  \right. \ .
\end{equation*}
From the second equation, it follows that $V^2=V^2(x^1)$. Integrating the first equation with respect to $x^1$, we get
$$V^1(x^1,x^2)=k_2\left(\frac{\displaystyle f_1'}{\displaystyle f_1^2}\right)(x^2) \ v^2(x^1)+G(x^2),$$
where $(v^2)'=V^2$ and $G$ is a smooth real function on $\mathbb R^2$ depending only on  $x^2$. Now, taking the derivative of $V^1$ with respect to $x^2$, we get
$$\frac{\displaystyle \partial V^1}{\displaystyle \partial x^2}(x^1,x^2)=k_2\left(\frac{\displaystyle f_1'}{\displaystyle f_1^2}\right)'(x^2) \ v^2(x^1)+G'(x^2),$$
which, replaced in the last equation of the system, gives
$$(V^2)'(x^1)=-k_2^2\frac{\displaystyle 1}{\displaystyle f_1^2(x^2)}\left(\frac{\displaystyle f_1'}{\displaystyle f_1}\right)'(x^2) \ v^2(x^1)-k_2
\left(\frac{\displaystyle (f_1G)'}{\displaystyle f_1^2}\right)(x^2).$$
Therefore,
$$(f_1G)'(x^2)=\frac{\displaystyle f_1^2(x^2)}{\displaystyle k_2}\left[-(V^2)'(x^1)-k_2^2\frac{\displaystyle 1}{\displaystyle f_1^2(x^2)}\left(\frac{\displaystyle f_1'}{\displaystyle f_1}\right)'(x^2) \ v^2(x^1)\right],$$
and we get, by integration with respect to $x^2$, that
$$G(x^2)=\frac{\displaystyle 1}{\displaystyle f_1(x^2)}\left[-\frac{\displaystyle F_1(x^2)(V^2)'(x^1)}{\displaystyle k_2}-k_2\left(\frac{\displaystyle f_1'}{\displaystyle f_1}\right)(x^2) \ v^2(x^1)+F(x^1)\right],$$
where $F_1'=f_1^2$ and $F$ is a smooth real function on $\mathbb R^2$ depending only on $x^1$. Then
$$V^1(x^1,x^2)=-\frac{\displaystyle F_1(x^2)}{\displaystyle k_2f_1(x^2)}(V^2)'(x^1)+\frac{\displaystyle F(x^1)}{\displaystyle f_1(x^2)},$$
and, by derivating $V^1$ with respect to $x^1$ and replacing it in the first equation of the system, implies that
$$F'(x^1)=k_2\frac{f_1'(x^2)}{f_1(x^2)}V^2(x^1)+\frac{F_1(x^2)}{k_2}(V^2)''(x^1).$$
Now, derivating $(V^2)'$, we find
$$(V^2)''(x^1)=-k_2^2\frac{\displaystyle 1}{\displaystyle f_1^2(x^2)}\left(\frac{\displaystyle f_1'}{\displaystyle f_1}\right)'(x^2) \ V^2(x^1),$$
therefore, $-k_2^2\frac{\displaystyle 1}{\displaystyle f_1^2}\left(\frac{\displaystyle f_1'}{\displaystyle f_1}\right)'$ must be a constant (let's say $k$), or $V^2=0$. 

In the first situation, we have obtained:
\begin{equation*}
\left\{
    \begin{aligned}
&(V^2)''=kV^2\\
&k=-k_2^2\frac{\displaystyle 1}{\displaystyle f_1^2}\left(\frac{\displaystyle f_1'}{\displaystyle f_1}\right)'
\end{aligned}
  \right. \ .
\end{equation*}
For the first equation of this system, the associated characteristic equation is $y^2-k=0$, and we have the following cases.

\noindent (a) If $k=0$ (equivalent to $\frac{\displaystyle f_1'}{\displaystyle f_1}=c_1\in \mathbb R$, i.e., $f_1(x^2)=c_2e^{c_1x^2}$, where $c_2\in \mathbb R\setminus \{0\}$), then
$$V^2(x^1)=c_3x^1+c_4, \ \ c_3, c_4\in \mathbb R.$$

\noindent (b) If $k>0$, then
$$V^2(x^1)=c_1e^{\sqrt{k}x^1}+c_2e^{-\sqrt{k}x^1}, \ \ c_1,c_2\in \mathbb R.$$

\noindent (c) If $k<0$, then
$$V^2(x^1)=c_1\cos({\sqrt{-k}x^1})+c_2\sin({\sqrt{-k}x^1}), \ \ c_1, c_2\in \mathbb R.$$

If $V^2=0$, then $F=c_0$, $c_0\in \mathbb R$, $G=\frac{\displaystyle c_0}{\displaystyle f_1}$, and $V^1(x^2)=\frac{\displaystyle c_0}{\displaystyle f_1(x^2)}$. 

The converse implication also holds true.
\end{proof}

\begin{example}
$V=-\frac{\displaystyle e^{x^2}}{\displaystyle 2}\frac{\displaystyle \partial}{\displaystyle \partial x^1}+x^1\frac{\displaystyle \partial}{\displaystyle \partial x^2}$ is a Killing vector field on
$$\left(\mathbb R^2, \ g=e^{-2x^2}dx^1\otimes dx^1+dx^2\otimes dx^2\right).$$
\end{example}

\begin{remark}
Let us notice from Theorem \ref{g6} that $\sign(k)=-\sign \left(\frac{\displaystyle f_1'}{\displaystyle f_1}\right)'$, and
$$k=-k_2^2\frac{\displaystyle 1}{\displaystyle f_1^2}\left(\frac{\displaystyle f_1'}{\displaystyle f_1}\right)'=-k_2^2\frac{\displaystyle f_1''f_1-(f_1')^2}{\displaystyle f_1^4};$$
therefore, $\frac{\displaystyle f_1''f_1-(f_1')^2}{\displaystyle f_1^4}$ must be constant. In particular, this condition is satisfied if $\left(\frac{\displaystyle f_1'}{\displaystyle f_1}\right)'=0$; in this case, $f_1(x^2)=c_1e^{c_2x^2}$, where $c_1\in \mathbb R\setminus \{0\}$ and $c_2\in \mathbb R$.
\end{remark}

If both of the functions $f_1$ and $f_2$ are constant, from the above theorems, we get

\begin{corollary}
If $f_1=k_1\in \mathbb R\setminus \{0\}$ and $f_2=k_2\in \mathbb R\setminus \{0\}$, then $V=V^1E_1+V^2E_2$ is a Killing vector field on $(\mathbb R^2,g)$ if and only if
$$V^1(x^2)=-\frac{k}{k_2}x^2+c_1 \ \ \textrm{and} \ \ V^2(x^1)=\frac{k}{k_1}x^1+c_2,$$
where $k,c_1,c_2\in \mathbb R$.
\end{corollary}
\begin{proof}
In this case, we have
\begin{equation*}
\left\{
    \begin{aligned}
&\frac{\displaystyle \partial V^1}{\displaystyle \partial x^1}=\frac{\displaystyle \partial V^2}{\displaystyle \partial x^2}=0\\
&k_1\frac{\displaystyle \partial V^2}{\displaystyle \partial x^1}=-k_2\frac{\displaystyle \partial V^1}{\displaystyle \partial x^2}
\end{aligned}
  \right. \ .
\end{equation*}
It follows that
$$V^1=V^1(x^2), \ \ V^2=V^2(x^1), \ \ \textrm{and} \ \ k_1(V^2)'=-k_2(V^1)' = \textrm{constant},$$
and we get the conclusion. 

The converse implication also holds true.
\end{proof}

\begin{remark}
If $f_1=f_2=1$, then it is well-known that the Killing vector fields $V$ on $\mathbb R^2$ with respect to the canonical metric $g_0=dx^1\otimes dx^1+dx^2\otimes dx^2$ are of the form
$$V_{(x^1,x^2)}=(c_0x^2+c_1)\left(\frac{\displaystyle \partial }{\displaystyle \partial x^1}\right)_{(x^1,x^2)}-(c_0x^1+c_2)\left(\frac{\displaystyle \partial }{\displaystyle \partial x^2}\right)_{(x^1,x^2)},$$
where $c_0,c_1,c_2\in \mathbb R$.
\end{remark}

\textit{Department of Mathematics}

\textit{Faculty of Mathematics and Computer Science}

\textit{West University of Timi\c{s}oara}

\textit{Bd. V. P\^{a}rvan 4, 300223, Timi\c{s}oara, Romania}

\textit{adarablaga@yahoo.com}

\end{document}